\documentclass{amsart}

\usepackage{amsmath,amsthm,amssymb}
\usepackage{hyperref}
\usepackage{mathrsfs}
\usepackage{dsfont}
\usepackage{enumerate}
\usepackage{graphicx}

\allowdisplaybreaks[3]

\newtheorem{lemma}{Lemma}[section]
\newtheorem{theorem}{Theorem}[section]
\newtheorem{corollary}{Corollary}[section]

\newtheorem{remark}{Remark}[section]

\numberwithin{equation}{section}

\DeclareMathOperator*{\osc}{osc}

\newcommand{\sref}[1]{Section~\ref{#1}}
\newcommand{\ssref}[1]{Subsection~\ref{#1}}
\newcommand{\tref}[1]{\textsl{Theorem~\ref{#1}}}
\newcommand{\lref}[1]{\textsl{Lemma~\ref{#1}}}
\newcommand{\cref}[1]{\textsl{Corollary~\ref{#1}}}

\newcommand{\eref}[1]{\textsc{(\ref{#1})}}

\newcommand{\norm}[1]{\left\Vert#1\right\Vert}

\newcommand{\R}{\mathbb R}
\newcommand{\Z}{\mathbb Z}

\newcommand{\MP}{\mathcal P}

\newcommand{\ve}{\varepsilon}

\newcommand{\ol}{\overline}

\newcommand{\p}{\partial}

\newcommand{\wt}{\widetilde}
\newcommand{\ra}{\rightarrow}

\begin{document}

\title[Global $W^{2,\delta}$ estimates for singular elliptic equations]
{Global $W^{2,\delta}$ estimates for singular
fully nonlinear elliptic equations
with $L^n$ right hand side terms}

\author{Dongsheng Li}
\address{Dongsheng Li: School of Mathematics and Statistics,
Xi'an Jiaotong University, Xi'an 710049, China;}
\email{\tt lidsh@mail.xjtu.edu.cn}

\author{Zhisu Li}
\address{Zhisu Li: School of Mathematics and Statistics,
Xi'an Jiaotong University, Xi'an 710049, China;}
\email{\tt lizhisu@stu.xjtu.edu.cn}

\thanks{%
This research is supported by NSFC.11671316.
Zhisu Li is the corresponding author.
}

\date{\today}


\begin{abstract}
We establish in this paper
\emph{a priori} global $W^{2,\delta}$ estimates
for singular fully nonlinear elliptic equations
with $L^n$ right hand side terms.
The method is to slide
paraboloids and barrier functions vertically
to touch the solution of the equation,
and then to estimate the measure of the contact set
in terms of the measure of the vertex point set.
To derive global estimates from $L^n$ data,
the Hardy-Littlewood maximal functions, appropriate localizations
and a new type of covering argument are adopted.
These methods also provide us a more direct proof
of the $W^{2,\delta}$ estimates
for (nonsingular) fully nonlinear elliptic equations
established by L. A. Caffarelli and X. Cabr\'{e}.
\end{abstract}


\maketitle

\section{Introduction}

In the present paper,
we derive \emph{a priori} global $W^{2,\delta}$ estimates
for solutions of the singular elliptic equations
including those of the following types:
\begin{enumerate}[\quad(a)]
\item
the equation
\begin{equation}\label{eq.trAD}
\mathrm{tr}(A(x)D^2u)+b(x)\cdot Du=f(x)|Du|^{\gamma},
\end{equation}
where $0\leq\gamma<1$, $A$ is uniformly elliptic,
$b$ is bounded and $f\in L^n$;

\item
the singular fully nonlinear elliptic equation
\begin{equation}\label{eq.DugaF}
|Du|^{-\gamma}F(D^2u,Du,u,x)=f(x),
\end{equation}
where $0\leq\gamma<1$,
$F(0,0,\cdot,\cdot)\equiv0$,
$F$ is uniformly elliptic
(see \cite{CC}), and $f\in L^n$;

\item
the famous $p$-Laplace equation
\begin{equation}\label{eq.ple}
\Delta_p u:=\mbox{div}\left(|Du|^{p-2}Du\right)=f(x),
\end{equation}
where $1<p\leq2$ and $f\in L^n$.
\end{enumerate}

For brevity, we consider solutions
of singular fully nonlinear elliptic inequalities of certain type
which include solutions of all the above equations.
Namely, our main result will be stated
in a more generalized form as follows.
\begin{theorem}\label{th.w2d}
Let $0<\lambda\leq\Lambda<+\infty$ and $0\leq\gamma<1$.
Suppose $u\in C^0(\ol{B_1})$ is a viscosity solution
of the singular fully nonlinear elliptic inequalities
\begin{equation}\label{eq.ineqs}
|Du|^{-\gamma}\MP_{\lambda,\Lambda}^-(D^2u)-|Du|^{1-\gamma}\leq f
\leq |Du|^{-\gamma}\MP_{\lambda,\Lambda}^+(D^2u)+|Du|^{1-\gamma}
~~\mathrm{in}~B_1,
\end{equation}
where $B_1$ is the unit open ball of $\R^n$,
$\MP_{\lambda,\Lambda}^\pm$
are the Pucci extremal operators,
and $f\in C^0\cap L^n(B_1)$.
Then $u\in W^{2,\delta}(B_1)$ and
\begin{equation}\label{eq.w2d-est}
\norm{u}_{W^{2,\delta}(B_1)}
\leq C\left(\norm{u}_{L^{\infty}(B_1)}
+\norm{f}_{L^n(B_1)}^{\frac{1}{1-\gamma}}\right)
\end{equation}
for any $\delta\in(0,\sigma)$,
where $\sigma=\sigma(n,\lambda,\Lambda,\gamma)>0$
and $C=C(n,\lambda,\Lambda,\gamma,\delta)>0$.
\end{theorem}
This theorem improves essentially our previous results in \cite{Li},
where the right hand side term $f$ of the equation
is required to be $L^\infty$,
and the estimate corresponding to \eref{eq.w2d-est} is just
\[\norm{u}_{W^{2,\delta}(B_1)}
\leq C\left(\norm{u}_{L^{\infty}(B_1)}
+\norm{f}_{L^\infty(B_1)}^{\frac{1}{1-\gamma}}\right).\]
The main contribution here is
in developing a systematic way to deal with the $L^n$ data,
and in using delicate localization and covering arguments
to derive global estimates from it in a straightforward way.
Roughly speaking, first, by sliding paraboloids and
some appropriate localizing barrier functions
from below and above to touch the solution,
and then estimating the low bound of
the measure of the set of contact points
by the measure of the set of vertex points,
we establish a new density estimate which is corresponding to
the classical Alexandroff-Bakelman-Pucci (ABP for short) estimate;
then, applying a new kind of covering technique with careful localization,
we obtain the desired global $W^{2,\delta}$ estimates.
These methods also provide us a more direct proof
of the interior $W^{2,\delta}$ estimates
for (nonsingular) fully nonlinear elliptic equations
established by L. A. Caffarelli and X. Cabr\'{e} \cite{CC}
(which now is recovered by \tref{th.w2d} as special cases,
rather than by our previous results in \cite{Li}),
although the underlying key ideas are the same.

The sliding paraboloid argument we mentioned above
has originated in the work of X. Cabr{\'e} \cite{Cab}
and continued in the work of O. Savin \cite{Sav},
see also \cite{IS}, \cite{CF} and \cite{Dan}.
We now give some other historical remarks concerning
the $W^{2,\delta}$ estimates and
the singular elliptic equations.

In 1986, F.-H. Lin \cite{Lin} first established
the $W^{2,\delta}$ estimates
$\norm{D^2u}_{L^\delta(B_1)}\leq C\norm{f}_{L^n(B_1)}$
for solutions of the linear uniformly elliptic equations
$a_{ij}(x)u_{ij}=f(x)$ with $u=0$ on $\p B_1$
and $f\in L^n(B_1)$. His method employs
the Fabes-Stroock type reverse H{\"o}lder inequality,
estimates for Green's function and the ABP estimate.
Later, L. A. Caffarelli and X. Cabr\'{e} \cite{CC}
(see also \cite{Caf}) applied ABP estimate,
Calder\'{o}n-Zygmund cube decomposition technique,
barrier function method and touching by tangent paraboloid method
to obtain interior $W^{2,\delta}$ estimates
for viscosity solutions of the fully nonlinear elliptic inequalities
\[\MP_{\lambda,\Lambda}^-(D^2u)\leq f
\leq \MP_{\lambda,\Lambda}^+(D^2u),\]
which can be viewed as an original form
of our inequalities \eref{eq.ineqs}.
As we know, the $W^{2,\delta}$ estimates
have several important applications
in the study of the elliptic partial differential equations,
such as deriving $W^{2,p}$ estimates
($p$ is large, see \cite{Caf} or \cite{CC}),
proving partial regularity
(see \cite{ASS} or \cite{Dan}),
and exploring the convergence of blow down solutions
(see \cite{Yuan}), and so on.

The investigation of singular elliptic equations
of the types \eref{eq.trAD} and \eref{eq.DugaF}
has made much progress in recent years.
The corresponding comparison principle (see \cite{BD1}),
ABP estimate (see \cite{DFQ1}),
Harnack inequality (see \cite{DFQ2})
and $C^{1,\alpha}$ estimate (see \cite{BD2})
have already been established.
To derive $W^{2,\delta}$ estimate for singular equations,
one may naturally think that it can be an easy consequence
of the classical results \cite{CC},
once we have a universal control of $\norm{Du}_{L^\infty}$,
for instance, some $C^{1,\alpha}$ estimate, like \cite{BD2}.
But this is always highly restricted and sometimes fails to be valid.
Our method can deal with a large class of equations
as illustrated above, since it does not depend on
any \emph{a priori} estimate of $Du$
and it does not use maximum principles.
Moreover, our estimates are global,
and the proof is straightforward
in the sense that we do not need
to separate it into interior estimates and boundary estimates.
Instead, the singular case and the nonsingular case,
the interior case and the boundary case,
are treated all together
by using delicate localization and a new type of covering lemma.

For $W^{2,\delta}$ estimate of
the singular $p$-Laplace equation \eref{eq.ple},
P. Tolksdorf \cite{Tol} proved that
each $W^{1,p}\cap C^0(B_1)$ weak solution
of \eref{eq.ple} in $B_1$ with $f\in L^{\infty}(B_1)$
is $W^{2,p}_{loc}\cap W^{1,p+2}_{loc}(B_1)$.
Since the $p$-Laplacian can be written as
\[
\Delta_p u=|Du|^{-(2-p)}
\left(\delta_{ij}-(2-p)\cdot\frac{D_{i}uD_{j}u}{|Du|^2}\right)D_{ij}u,
\]
applying our \tref{th.w2d} to
the singular $p$-Laplace equation \eref{eq.ple}
with $f\in L^n(B_1)$,
we obtain a new global $W^{2,\delta}$ estimate.

\medskip

The paper is organized as follows.
In \sref{se.pre}, we give some notations
and collect some preliminary lemmas
including a new type of covering lemma.
In \sref{se.gw2de} we first normalize \tref{th.w2d} to
\lref{le.w2dc} by rescaling argument in \ssref{sse.th-normto-le},
then in \ssref{sse.keylem},
we establish the key density lemma
and the measure decay estimate lemma,
with the help of them
we finally give the proof of \lref{le.w2dc}
in \ssref{sse.pf-le}.

\section{Some preliminaries}\label{se.pre}

In this paper,
we denote by $S(n)$ the linear space
of symmetric $n\times n$ real matrices
and $I$ the identity matrix.

\medskip

To make the \emph{sliding and touching} idea more rigorous and clear,
we introduce the following notations and terminologies.

Given two functions
$u$ and $v:\Omega\subset\R^n\ra\R$
and a point $x_0\in\Omega$,
we say that
\emph{$u$ touches $v$ by below at $x_0$ in $\Omega$}
and denote it briefly by \emph{$u\overset{x_0}\eqslantless v$ in $\Omega$}, if
$u(x_0)=v(x_0)$ and $u(x)\leq v(x)$, $\forall x\in\Omega$.

For a given continuous function
$u:U\subset\R^n\ra\R$,
we slide the concave paraboloid
(of opening $\kappa>0$ and of vertex $y$)
\[-\frac{\kappa}{2}|x-y|^2+C\]
vertically from below in $U$ (by increasing or decreasing $C$)
till it touches the graph of $u$ for the first time.
If the contact point is $x_0$,
we then have
\[C=u(x_0)+\frac{\kappa}{2}|x_0-y|^2
=\underset{x\in U}\inf\left(u(x)+\frac{\kappa}{2}|x-y|^2\right).\]

Given a closed set $V\subset\R^n$
and a continuous function $u:B_1\ra\R$,
we now introduce the definitions of \emph{the contact sets} as follows:
\begin{eqnarray}\label{eq.Tk-def}
T_\kappa^{-}(V)&:=&T_\kappa^{-}(u,V)
:=\bigg\{x_0\in B_1\mid\exists y\in V~\mbox{such~that}~\nonumber\\
&~&~~u(x_0)+\frac{\kappa}{2}|x_0-y|^2
=\underset{x\in B_1}\inf\left(u(x)+\frac{\kappa}{2}|x-y|^2\right)\bigg\}\nonumber\\
&=&~\bigg\{x_0\in B_1\mid\exists y\in V~\mbox{such~that}\nonumber\\
&~&~~-\frac{\kappa}{2}|x-y|^2+u(x_0)+\frac{\kappa}{2}|x_0-y|^2
~\overset{x_0}\eqslantless u~\mbox{in}~B_1\bigg\},
\end{eqnarray}
\[T_\kappa^+(V):=T_\kappa^+(u,V):=T_\kappa^-(-u,V),\]
and
\[T_\kappa(V):=T_\kappa(u,V):=T_\kappa^-(u,V)\cap T_\kappa^+(u,V).\]
For simplicity, we will write $T_\kappa^\pm$
instead of $T_\kappa^\pm(u,\ol{B_1})$ when there is no confusion.
It is obvious that $T_\kappa^\pm$ are closed in $B_1$.
We remark further that the contact set $T^-_\kappa(u,V)$
has the twofold uses of $\{u=\Gamma_u\}$
and $\underline{G}_M(u,\Omega)$ in \cite{CC}:
by the former, we communicate with the equation;
by the later, we measure the second derivatives of the solution.

\medskip

Given $0<\lambda\leq\Lambda$, we define
the \emph{Pucci extremal operators} (see also \cite{CC}) by
\[\MP^+_{\lambda,\Lambda}(X)
:=\lambda\sum_{e_i(X)<0}e_i(X)+\Lambda\sum_{e_i(X)>0}e_i(X),\]
and
\[\MP^-_{\lambda,\Lambda}(X)
:=\Lambda\sum_{e_i(X)<0}e_i(X)+\lambda\sum_{e_i(X)>0}e_i(X),\]
where $X\in S(n)$ and $e_i(X)$ denote the eigenvalues of $X$.
For brevity, we will always
write $\MP^{\pm}_{\lambda,\Lambda}(X)$ as $\MP^{\pm}(X)$.
For completeness and convenience,
we now collect some basic properties
of the Pucci extremal operators as follows:
\begin{enumerate}[(i)]
\item
$\MP^\pm(rX)=r\MP^\pm(X)$,
$\MP^\pm(-rX)=-r\MP^\mp(X)$,
$\forall X\in S(n)$, $\forall r\geq0$.
\item
$\MP^-(X)+\MP^-(Y)\leq \MP^-(X+Y)
\leq \MP^-(X)+\MP^+(Y)
\leq \MP^+(X+Y)$ $\leq \MP^+(X)+\MP^+(Y)$,
$\forall X,Y\in S(n)$.
\item
If $X,Y\in S(n)$ and $X\leq Y$,
then $\MP^\pm(X)\leq \MP^\pm(Y)$.
\item
$\MP^+_{\lambda,\Lambda}(X)=\Lambda\mathrm{tr}X$
and $\MP^-_{\lambda,\Lambda}(X)=\lambda\mathrm{tr}X$,
provided $X\in S(n)$ and $X\geq0$.
\end{enumerate}

\medskip

Now we recall the definition of the viscosity solution (see \cite{CC}).
For example, we say that $u\in C^0(B_1)$ \emph{satisfies}
\[
F(D^2u,Du,u,x)\leq f~~\mbox{in}~B_1
\]
\emph{in the viscosity sense},
if $\forall\varphi\in C^2(B_1)$, $\forall x_0\in B_1$,
\[
\varphi\overset{x_0}\eqslantless u~\mbox{in}~U(x_0)
\Rightarrow
F\left(D^2\varphi(x_0),D\varphi(x_0),\varphi(x_0),x_0\right)\leq f(x_0),
\]
where $U(x_0)\subset B_1$ is an open neighborhood of $x_0$.

\bigskip

For $g\in L^1(\Omega)$,
the \emph{Hardy-Littlewood maximal function} of $g$
is defined by
\[\mathcal{M}(g)(x):=\underset{r>0}\sup\frac{1}{|B_r(x)|}
\underset{B_r(x)\cap\Omega}\int|g(y)|dy,~
\forall~x\in\Omega.\]
We will use the well known weak type (1,1) property
of the Hardy-Littlewood maximal operator $\mathcal{M}$, that is
\[|\{x\in\Omega:\mathcal{M}(g)(x)>t\}|
\leq Ct^{-1}\norm{f}_{L^1(\Omega)},~\forall~t>0,\]
where $C=C(n)>0$ depends only on the dimension $n$.

\medskip

The following equivalent description
of $L^p$-integrability is also needed.
\begin{lemma}\label{le.lp} (see \cite[Lemma 7.3.]{CC})
Let g be a nonnegative and measurable function in a
bounded domain $\Omega\subset\R^n$.
Suppose that $\eta>0$, $M>1$ and $0<p<\infty$.
Then
\[g\in L^p(\Omega)~\Leftrightarrow~s
:=\sum_{k=1}^{\infty}M^{pk}
\left|\left\{x\in\Omega\mid g(x)>\eta M^k\right\}\right|<\infty,\]
and
\[C^{-1}s\leq\|g\|_{L^p(\Omega)}^p\leq C\left(s+|\Omega|\right),\]
where $C>0$ is a constant depending only on $\eta$, $M$ and $p$.\\
\end{lemma}

\medskip

Finally, we introduce the following
Vitali-type covering lemma modified from
those in \cite{IS} and \cite{Li}.
This lemma plays a similar role as
the Calder{\'o}n-Zygmund cube decomposition lemma
(see \cite{Caf} and \cite{CC}) usually does,
but with the help of it we can obtain global estimates directly.
\begin{lemma}\label{le.cl}
$\big((\theta, \Theta)$-$\mathrm{type~covering~lemma\big)}$.
Let $E\subset F\subset B_1$
be measurable sets and $0<\theta<\Theta<1$ such that
\begin{enumerate}[(i)]
\item $|E|>\theta|B_1|$, and
\item for any ball $B\subset B_1$, if $|B\cap E|\geq\theta|B|$,
then $|B\cap F|\geq\Theta|B|$.
\end{enumerate}
Then
\[|B_1\setminus F|\leq
\left(1-\frac{\Theta-\theta}{5^n}\right)|B_1\setminus E|.\]
\end{lemma}

\begin{proof}
It suffices to prove that
\[|F\setminus E|\geq
\frac{\Theta-\theta}{5^n}|B_1\setminus E|.\]

By the Lebesgue differentiation theorem,
there exists $S\subset B_1\setminus E$,
such that $|S|=|B_1\setminus E|$ and
\[\lim_{r\ra0}\frac{|B_r(x)\cap E|}{|B_r(x)|}=0,~\forall x\in S.\]
Hence, for each $x\in S$,
there exist balls, say $B$, satisfying
$x\in B\subset B_1$ and $|B\cap E|\leq\theta|B|$;
we choose one of the biggest of them and denote it by $B^x$.

We assert that $|B^x\cap F|\geq\Theta|B^x|$.
Otherwise, suppose that $|B^x\cap F|<\Theta|B^x|$.
Since $|B_1\cap E|>\theta|B_1|$,
$|B^x\cap E|\leq\theta|B^x|$
and hence $B^x\subsetneqq B_1$,
we may enlarge $B^x$ a little bit, denoted by $\wt{B}^x$,
such that $B^x\subset\wt{B}^x\subset B_1$
and $|\wt{B}^x\cap F|<\Theta|\wt{B}^x|$.
By the hypothesis \emph{(ii)} of the lemma,
$|\wt{B}^x\cap E|<\theta|\wt{B}^x|$,
which contradicts the definition of $B^x$.

Furthermore, since $|B^x\setminus E|\geq(1-\theta)|B^x|$,
it follows from the above assertion that
$|B^x\cap F\setminus E|\geq(\Theta-\theta)|B^x|$.

Now consider the covering
$\underset{x\in S}\cup B^x \supset S$.
By the Vitali covering lemma,
there exists an at most countable set of $x_i\in S$,
such that $\{B^{x_i}\}_{i}$ are disjoint
and $\underset{i}\cup 5B^{x_i}\supset S$.
Hence we have
\begin{eqnarray*}
|F\setminus E|
&\geq&\left|\left(\underset{i}\cup B^{x_i}\right)\cap F\setminus E\right|
=\left|\underset{i}\cup\left(B^{x_i}\cap F\setminus E\right)\right|
=\underset{i}\sum\left|B^{x_i}\cap F\setminus E\right|\\
&\geq&(\Theta-\theta)\underset{i}\sum\left|B^{x_i}\right|
=\frac{\Theta-\theta}{5^n}\underset{i}\sum\left|5B^{x_i}\right|
\geq\frac{\Theta-\theta}{5^n}\left|\underset{i}\cup 5B^{x_i}\right|\\
&\geq&\frac{\Theta-\theta}{5^n}|S|
=\frac{\Theta-\theta}{5^n}\left|B_1\setminus E\right|.
\end{eqnarray*}
This completes the proof of the lemma.
\end{proof}

\section{Proof of \tref{th.w2d}}\label{se.gw2de}

\subsection{\tref{th.w2d} can be normalized to \lref{le.w2dc}}\label{sse.th-normto-le}
\quad

To prove \tref{th.w2d}, it suffices to prove the follow lemma.
\begin{lemma}\label{le.w2dc}
Let $0\leq\gamma<1$.
Assume that
$u\in C^0(\ol{B_1})$
satisfies (\ref{eq.ineqs})
with $f\in C^0\cap L^n(B_1)$ in the viscosity sense.
Then there exist constants
$\sigma=\sigma(n,\lambda,\Lambda,\gamma)>0$
and $\epsilon_1=\epsilon_1(n,\lambda,\Lambda)>0$,
such that for any $\delta\in(0,\sigma)$,
if $\norm{u}_{L^{\infty}(B_1)}\leq 1/16$
and $\norm{f}_{L^n(B_1)}\leq \epsilon_1$,
then
\[\norm{u}_{W^{2,\delta}(B_1)}\leq C,\]
where $C=C(n,\lambda,\Lambda,\gamma,\delta)>0$.
\end{lemma}

Indeed, suppose $u$ satisfies
the hypothesis of \tref{th.w2d}.
Let
\[\alpha:=\left(16\norm{u}_{L^\infty(B_1)}
+\left(\epsilon_1^{-1}\norm{f}_{L^n(B_1)}\right)^{\frac{1}{1-\gamma}}
+\ve\right)^{-1}\]
for any $\ve>0$.
Then the scaled function $\wt{u}(x):=\alpha u(x)$ solves
\[|D\wt{u}|^{-\gamma}\MP_{\lambda,\Lambda}^-(D^2\wt{u})
-|D\wt{u}|^{1-\gamma}
\leq \alpha^{1-\gamma}f=:\wt{f}
\leq |D\wt{u}|^{-\gamma}\MP_{\lambda,\Lambda}^+(D^2\wt{u})
+|D\wt{u}|^{1-\gamma}
~~\mathrm{in}~B_1,\]
and satisfies $\norm{\wt{u}}_{L^\infty(B_1)}\leq 1/16$
and $\norm{\wt{f}}_{L^n(B_1)}\leq \epsilon_1$.
By \lref{le.w2dc}, we have
\[\norm{\wt{u}}_{W^{2,\delta}(B_1)}\leq C=C(n,\lambda,\Lambda,\gamma,\delta).\]
Scaling back to $u$ and letting $\ve\ra0$, we obtain
\[
\norm{u}_{W^{2,\delta}(B_1)}\leq C\alpha^{-1}
\leq C\left(\|u\|_{L^{\infty}(B_1)}+\|f\|_{L^n(B_1)}^{\frac{1}{1-\gamma}}\right).
\]
\tref{th.w2d} thereby is proved.

\subsection{Key lemmas for the proof of \lref{le.w2dc}}\label{sse.keylem}
\quad

\lref{le.w2dc} will be established via several lemmas of this subsection.
The most important one is the density estimate lemma, \lref{le.dl},
which is a key lemma in this paper
and can be viewed as a measure theoretic ABP estimate.
The strategy for the proof of \lref{le.dl}
is modified from those in \cite{Sav}, \cite{CF} and \cite{Li}.
Since the right hand side term $f$ belongs to $L^n$,
the Hardy-Littlewood maximal functions
and certain careful localization techniques
have to be employed here.
Note also that, in order to obtain global regularity,
it is crucial to show that the contact sets
are contained in the interior of $B_1$,
in the proof of the following lemma.
\begin{lemma}\label{le.dl}
Let $0\leq\gamma<1$, $0<\theta_0<1$ and $K\geq1$.
Assume that $u\in C^0(\ol{B_1})$ satisfies
\begin{equation}\label{eq.sups}
|Du|^{-\gamma}\MP_{\lambda,\Lambda}^-(D^2u)-|Du|^{1-\gamma}
\leq f~~\mbox{in}~B_1
\end{equation}
in the viscosity sense,
where $f\in C^0\cap L^n(B_1)$.
Then there exist constants
$\epsilon_2=\epsilon_2(n,\lambda,\Lambda,\theta_0)>0$,
$M=M(n,\lambda,\Lambda)>1$,
$0<\Theta=\Theta(n,\lambda,\Lambda)<1$ and
$0<\theta=\theta(n,\lambda,\Lambda,\theta_0)
<\min\{\theta_0,\Theta\}<1$ such that
if $B_r(x_0)\subset B_1$ satisfies
\begin{equation}\label{eq.gt}
\left|B_r(x_0)\cap T^-_K\cap
\{x\in B_1:\mathcal{M}(|f|^n)(x)\leq\epsilon_2 K^{(1-\gamma)n}\}\right|
\geq\theta\left|B_r(x_0)\right|,
\end{equation}
then
\[\left|B_r(x_0)\cap T^-_{KM}\right|\geq\Theta\left|B_r(x_0)\right|.\]
(For the definitions of $T^-_{K}$ and $T^-_{KM}$,
see \eref{eq.Tk-def} in \sref{se.pre}.)
\end{lemma}

\begin{proof}
For simplicity, we may assume
without loss of generality that $u$ is smooth in $B_1$.
Otherwise one needs to
regularize $u$ using the standard $\epsilon$-envelope
method of Jensen (see for instance \cite{CC}, \cite{Sav} and \cite{Li}).

\medskip

Take $0<\theta<1$ to be chosen later.
It follows from \eref{eq.gt} that there exists
$0<\ve=\ve(n,\theta)<1$
such that
\[B_{(1-\ve)r}(x_0)\cap T^-_K\cap\{x\in B_1:\mathcal{M}(|f|^n)(x)
\leq\epsilon_2 K^{(1-\gamma)n}\}\neq\emptyset.\]
Let $x_1\in B_{(1-\ve)r}(x_0)
\cap T^-_K\cap\{x\in B_1:\mathcal{M}(|f|^n)(x)\leq\epsilon_2K^{(1-\gamma)n}\}$.
Then, by the definition \eref{eq.Tk-def} of $T^-_K$,
there exists $ y_1\in \ol{B_1}$ such that
\begin{equation}\label{eq.PK-x1y1-u}
P^-_{K,y_1}(x)
:=-\frac{K}{2}|x-y_1|^2+u(x_1)
+\frac{K}{2}|x_1-y_1|^2\overset{x_1}\eqslantless u~~\mbox{in}~B_1.
\end{equation}
Note that $B_{3r/4}(x_1)\cap B_{r/4}(x_0)$ contains a ball of radius $\ve r/2$.
The proof now will be split into three steps.

\medskip

\emph{Step 1.}
We prove that there exist $x_2\in B_{r/2}(x_0)$ and $C=C(n,\lambda,\Lambda)>0$
such that
\begin{equation}\label{eq.x2c}
u(x_2)-P^-_{K,y_1}(x_2)\leq CKr^2.
\end{equation}

To do this, for each $y_{\scriptscriptstyle V}
\in B_{3r/4}(x_1)\cap B_{r/4}(x_0)$($\subset B_1$),
we set
\[\psi(x):=P^-_{K,y_1}(x)
+K\rho ^2\varphi\left(\frac{x-y_{\scriptscriptstyle V}}{\rho}\right),\]
where $\rho =3r/4$, $\varphi(X)=\phi(|X|)$
and $\phi(t)=e^A e^{-At^2}-1$ with $A>1$ to be determined later.
Let
$x_{\scriptscriptstyle T}
\in \ol{B_{\rho}(y_{\scriptscriptstyle V})}$
such that
\[(u-\psi)(x_{\scriptscriptstyle T})
=\underset{\ol{B_{\rho}(y_{\scalebox{0.35} V})}}\min(u-\psi).\]
Since \eref{eq.PK-x1y1-u} implies
$(u-\psi)|_{\partial B_{\rho}(y_{\scalebox{0.35} V})}\geq0$
and
\[(u-\psi)(x_{\scriptscriptstyle T})\leq(u-\psi)(x_1)
=-K\rho ^2\varphi\left(\frac{x_1-y_{\scriptscriptstyle V}}{\rho}\right)<0,\]
we conclude that
$x_{\scriptscriptstyle T}\in B_{\rho}(y_{\scriptscriptstyle V})
\subset B_r(x_0)\subset B_1$
and
\begin{equation}\label{eq.x2}
u(x_{\scriptscriptstyle T})<\psi(x_{\scriptscriptstyle T})
=P^-_{K,y_1}(x_{\scriptscriptstyle T})
+K\rho ^2\varphi
\left(\frac{x_{\scriptscriptstyle T}-y_{\scriptscriptstyle V}}{\rho}\right)
\leq P^-_{K,y_1}(x_{\scriptscriptstyle T})+e^AKr^2.
\end{equation}

\medskip

We assert that:
$\exists A^0=A^0(n,\lambda,\Lambda)>1$,
$\exists y_{\scriptscriptstyle V}^0\in B_{3r/4}(x_1)\cap B_{r/4}(x_0)$
and
$\exists x_{\scriptscriptstyle T}^0\in B_{\rho /4}(y_{\scriptscriptstyle V}^0)$
such that
\begin{equation}\label{eq.xt0}
(u-\psi)(x_{\scriptscriptstyle T}^0)
=\underset{\ol{B_{\rho}(y_{\scalebox{0.35} V}^0)}}\min (u-\psi).
\end{equation}
Once we have proved \eref{eq.xt0}, it will follow that
\[|x_{\scriptscriptstyle T}^0-x_0|\leq|x_{\scriptscriptstyle T}^0-y_{\scriptscriptstyle V}^0|+|y_{\scriptscriptstyle V}^0-x_0|<3r/16+r/4<r/2.\]
In view of \eref{eq.x2},
this establishes \eref{eq.x2c}
by setting $x_2:=x_{\scriptscriptstyle T}^0$ and $C:=e^A$.

\medskip

Hence we now need only to prove \eref{eq.xt0}.
Suppose by contradiction that
$\forall A>1$,
$\forall y_{\scriptscriptstyle V}\in B_{3r/4}(x_1)\cap B_{r/4}(x_0)$
and $\forall x_{\scriptscriptstyle T}\in B_{\rho}(y_{\scriptscriptstyle V})$,
if they satisfy
\begin{equation}\label{eq.yvxt}
(u-\psi)(x_{\scriptscriptstyle T})
=\underset{\ol{B_{\rho}(y_{\scalebox{0.35} V})}}\min (u-\psi),
\end{equation}
then
\[x_{\scriptscriptstyle T}\in
B_{\rho}(y_{\scriptscriptstyle V})
\setminus B_{\rho /4}(y_{\scriptscriptstyle V}).\]
Let us set $t:=|x_{\scriptscriptstyle T}-y_{\scriptscriptstyle V}|/{\rho}$
and denote by $T$ the set of $x_{\scriptscriptstyle T}$
while the corresponding $y_{\scriptscriptstyle V}$ runs through
$B_{3r/4}(x_1)\cap B_{r/4}(x_0)$.
Then $1/4\leq t<1$ and $T\subset B_r(x_0)\subset B_1$.
The proof now will be divided into five minor steps.

$1^\circ$~~From \eref{eq.yvxt}, we see that
\[\psi+\underset{\ol{B_{\rho}(y_{\scalebox{0.35} V})}}\min (u-\psi)
\overset{x_{\scalebox{0.35} T}}\eqslantless u~~
\mathrm{in}~B_{\rho}(y_{\scriptscriptstyle V}).\]
By the definition of the viscosity solution of \eref{eq.sups}, we obtain
\begin{equation}\label{eq.xt}
|D\psi(x_{\scriptscriptstyle T})|^{-\gamma}
\MP_{\lambda,\Lambda}^-\left(D^2\psi(x_{\scriptscriptstyle T})\right)
-|D\psi(x_{\scriptscriptstyle T})|^{1-\gamma}
\leq f(x_{\scriptscriptstyle T}).
\end{equation}

Since
\[\left|DP^-_{K,y_1}(x_{\scriptscriptstyle T})\right|\leq CK,
\quad
\left|D^2P^-_{K,y_1}(x_{\scriptscriptstyle T})\right|\leq CK,\]
\[\left|\frac{\phi_{t}}{t}\right|\leq CAe^Ae^{-At^2}
\quad\mathrm{and}\quad
\left|\phi_{tt}\right|\geq C^{-1}A^2e^Ae^{-At^2},\]
we deduce that
\[|D\psi(x_{\scriptscriptstyle T})|\leq CKAe^A e^{-At^2},\]
\[|D\psi(x_{\scriptscriptstyle T})|^{\gamma}
\leq C\left(KAe^A e^{-At^2}\right)^{\gamma},\]
and
\begin{eqnarray*}
&~&\MP_{\lambda,\Lambda}^-(D^2\psi)
(x_{\scriptscriptstyle T})-|D\psi(x_{\scriptscriptstyle T})|\\
&\geq& C^{-1}K\left(|\phi_{tt}|-1\right)
-CK\left(\left|\frac{\phi_{t}}{t}\right|+1\right)-CKAe^A e^{-At^2}\\
&\geq& C^{-1}KAe^A e^{-At^2};
\end{eqnarray*}
and consequently
\[|D\psi(x_{\scriptscriptstyle T})|^{-\gamma}
\MP_{\lambda,\Lambda}^-\left(D^2\psi(x_{\scriptscriptstyle T})\right)
-|D\psi(x_{\scriptscriptstyle T})|^{1-\gamma}
\geq C^{-1}\left(KAe^A e^{-At^2}\right)^{1-\gamma}
\geq K^{1-\gamma},\]
where $A$ is sufficiently large
and all the $C$'s depend only on $n,\lambda$ and $\Lambda$.
Combining it with \eref{eq.xt}, we obtain
\begin{equation}\label{eq.kf}
0<1\leq K^{1-\gamma}\leq f(x_{\scriptscriptstyle T}),
~\forall x_{\scriptscriptstyle T}\in T.
\end{equation}

$2^\circ$~~Since the inflection point of $\phi(t)$
is $t=(2A)^{-1/2}$, we can assert that
\begin{equation}\label{eq.d1ph}
\left|D\varphi(X)\right|\leq C(n,\lambda,\Lambda),
\end{equation}
\begin{equation}\label{eq.d2ph}
\left|\left|D^2\varphi(X)\right|\right|\leq C(n,\lambda,\Lambda)
\end{equation}
and
\begin{equation}\label{eq.d2phinv}
\left|\left|\left(D^2\varphi(X)\right)^{-1}\right|\right|
\leq C(n,\lambda,\Lambda)
\end{equation}
for all $1/4\leq|X|<1$, provided $A$ is large enough.

$3^\circ$~~For any $x_{\scriptscriptstyle T}\in T$, we have
\begin{equation}\label{eq.du}
Du(x_{\scriptscriptstyle T})=D\psi(x_{\scriptscriptstyle T})
=-K(x_{\scriptscriptstyle T}-y_1)+K\rho D\varphi
\left(\frac{x_{\scriptscriptstyle T}-y_{\scriptscriptstyle V}}{\rho}\right),
\end{equation}
and
\[D^2u(x_{\scriptscriptstyle T})
\geq D^2\psi(x_{\scriptscriptstyle T})
=-KI+KD^2\varphi\left(\frac{x_{\scriptscriptstyle T}
-y_{\scriptscriptstyle V}}{\rho}\right).\]
Recalling \eref{eq.d1ph}, \eref{eq.d2ph} and \eref{eq.kf}, it follows that
\[|Du(x_{\scriptscriptstyle T})|\leq K|x_{\scriptscriptstyle T}-y_1|+K\rho
\left|D\varphi
\left(\frac{x_{\scriptscriptstyle T}-y_{\scriptscriptstyle V}}{\rho}\right)\right|
\leq CK,\]
and
\[D^2u(x_{\scriptscriptstyle T})
\geq-KI-K\left|\left|D^2\varphi
\left(\frac{x_{\scriptscriptstyle T}-y_{\scriptscriptstyle V}}{\rho}\right)
\right|\right|I
\geq-CKI\geq-CK^{\gamma}f(x_{\scriptscriptstyle T})I.\]
Combining them with \eref{eq.xt} and \eref{eq.kf}, we deduce that
\begin{eqnarray*}
\lambda\sum_{e_i\left(D^2u(x_{\scriptscriptstyle T})\right)>0}
e_i\left(D^2u(x_{\scriptscriptstyle T})\right)
&=&\MP_{\lambda,\Lambda}^-\left(D^2u(x_{\scriptscriptstyle T})\right)
-\Lambda\sum_{e_i\left(D^2u(x_{\scriptscriptstyle T})\right)<0}
e_i\left(D^2u(x_{\scriptscriptstyle T})\right)\\
&\leq&|Du(x_{\scriptscriptstyle T})|^{\gamma}f(x_{\scriptscriptstyle T})
+|Du(x_{\scriptscriptstyle T})|+CK^{\gamma}f(x_{\scriptscriptstyle T})\\
&\leq&(CK)^{\gamma}f(x_{\scriptscriptstyle T})
+CK+CK^{\gamma}f(x_{\scriptscriptstyle T})\\
&\leq& CK^{\gamma}f(x_{\scriptscriptstyle T}).
\end{eqnarray*}
Hence the absolute value
of each eigenvalue of $D^2u(x_{\scriptscriptstyle T})$
is not greater than $CK^{\gamma}f(x_{\scriptscriptstyle T})$,
and therefore
\begin{equation}\label{eq.diju}
\left|D_{ij}u(x_{\scriptscriptstyle T})\right|
\leq CK^{\gamma}f(x_{\scriptscriptstyle T})
\quad(\forall i,j=1,2,...,n),
\end{equation}
where $C=C(n,\lambda,\Lambda)>0$.

$4^\circ$~~Write
\[X:=\frac{x_{\scriptscriptstyle T}-y_{\scriptscriptstyle V}}{\rho}\]
and
\[Y:=\frac{Du(x_{\scriptscriptstyle T})+K(x_{\scriptscriptstyle T}-y_1)}{K\rho}.\]
From \eref{eq.du},
we have $D\varphi(X)=Y$, and hence $X=(D\varphi)^{-1}(Y)$.
By the inverse function theorem, we compute
\[D_YX=\left(D^2\varphi(X)\right)^{-1}
=\left(D^2\varphi\circ(D\varphi)^{-1}(Y)\right)^{-1},\]
and
\[D_{x_{\scalebox{0.35} T}}X=D_YX\cdot D_{x_{\scalebox{0.35} T}}Y
=\left(D^2\varphi(X)\right)^{-1}
\cdot\frac{D^2u(x_{\scriptscriptstyle T})+KI}{K\rho}.\]
Since
$y_{\scriptscriptstyle V}=x_{\scriptscriptstyle T}-\rho X$,
we obtain
\[D_{x_{\scalebox{0.35} T}}{y_{\scriptscriptstyle V}}
=I-\rho D_{x_{\scalebox{0.35} T}}X
=I-\left(D^2\varphi(X)\right)^{-1}
\cdot\frac{D^2u(x_{\scriptscriptstyle T})+KI}{K}.\]
Thus, in view of \eref{eq.d2phinv}, \eref{eq.diju} and \eref{eq.kf}, we have
\begin{eqnarray*}
\left|(D_{x_{\scalebox{0.35} T}}{y_{\scriptscriptstyle V}})_{(i,j)}\right|
&\leq&\left|\delta_{ij}-\underset{k}\sum\left(\left(D^2\varphi(X)\right)^{-1}_{(i,k)}
\cdot\frac{D_{kj}u(x_{\scriptscriptstyle T})+K\delta_{kj}}{K}\right)\right|\\
&\leq&1+\left|\left|\left(D^2\varphi(\cdot)\right)^{-1}
\right|\right|_{L^{\infty}\left(B_1\setminus B_{1/4}\right)}
\cdot\underset{k}\sum\left(1+\frac{D_{kj}u(x_{\scriptscriptstyle T})}{K}\right)\\
&\leq&1+C+CK^{-(1-\gamma)}f(x_{\scriptscriptstyle T})\\
&\leq& CK^{-(1-\gamma)}f(x_{\scriptscriptstyle T})
\quad(\forall i,j=1,2,...,n),
\end{eqnarray*}
and consequently
\begin{equation}\label{eq.detdy}
\left|\det{(D_{x_{\scalebox{0.35} T}}{y_{\scriptscriptstyle V}})}\right|
\leq C(n,\lambda,\Lambda)K^{-(1-\gamma)n}|f(x_{\scriptscriptstyle T})|^n.
\end{equation}

$5^\circ$~~Consider the mapping
$y_{\scriptscriptstyle V}:x_{\scriptscriptstyle T}\mapsto y_{\scriptscriptstyle V}$,
$T \ra B_{3r/4}(x_1)\cap B_{r/4}(x_0)$, given precisely by
$y_{\scriptscriptstyle V}=x_{\scriptscriptstyle T}-\rho X$.
By the area formula and \eref{eq.detdy}, we have
\[\left|B_{3r/4}(x_1)\cap B_{r/4}(x_0)\right|
\leq\underset{T}\int \left|\det{(D_{x_{\scalebox{0.35} T}}
{y_{\scriptscriptstyle V}})}\right|dx_{\scriptscriptstyle T}
\leq CK^{-(1-\gamma)n}\underset{T}\int
|f(x_{\scriptscriptstyle T})|^ndx_{\scriptscriptstyle T}.\]
Then, in light of the facts that
$B_{3r/4}(x_1)\cap B_{r/4}(x_0)$ contains a ball of radius $\ve r/2$,
$T\subset B_{r}(x_0)\subset B_{2r}(x_1)$
and $\mathcal{M}(|f|^n)(x_1)\leq\epsilon_2K^{(1-\gamma)n}$,
we can conclude that
\begin{eqnarray*}
2^{-n}\ve^n|B_1|r^n=|B_{\ve r/2}|
&\leq& CK^{-(1-\gamma)n}\underset{B_{2r}(x_1)}
\int|f(x_{\scriptscriptstyle T})|^ndx_{\scriptscriptstyle T}\\
&\leq&CK^{-(1-\gamma)n}\mathcal{M}(|f|^n)(x_1)|B_{2r}(x_1)|\\
&\leq& 2^nC|B_1|r^n\epsilon_2.
\end{eqnarray*}

Let $\epsilon_2<2^{-n}\ve^n\cdot(2^nC)^{-1}=4^{-n}C^{-1}\ve^n$,
we arrive at a contradiction.
This proves \eref{eq.xt0}
and completes the proof of the assertion \eref{eq.x2c}.
\medskip

\emph{Step 2.}
We now prove that there exists
$M=M(n,\lambda,\Lambda)>1$ such that
\begin{equation}\label{eq.tvbt}
T^-_{KM}(V)
\subset B_r(x_0)\cap T^-_{KM},
\end{equation}
where
\[V:=\ol{B_{\frac{M-1}{M}\cdot\frac{r}{8}}
\left(\frac{M-1}{M}x_2+\frac{1}{M}y_1\right)}.\]

For each $\wt{x}\in T^-_{KM}(V)$,
there exists $\wt{y}\in V$
such that
\begin{equation*}
P^-_{KM,\wt{y}}(x)
:=-\frac{KM}{2}|x-\wt{y}|^2+u(\wt{x})
+\frac{KM}{2}|\wt{x}-\wt{y}|^2
\overset{\wt{x}}\eqslantless u~~\mbox{in}~B_1.
\end{equation*}
Since
\[P^-_{KM,\wt{y}}(x)-P^-_{K,y_1}(x)
=-\frac{K(M-1)}{2}|x-y|^2+R,\]
where
\begin{equation}\label{eq.ydef}
y:=\frac{M}{M-1}\wt{y}-\frac{1}{M-1}y_1
\end{equation}
and $R=R\left(\wt{y},K,M,y_1,\wt{x},
u(\wt{x}),x_1,u(x_1)\right)$
both do not depend on $x$, we obtain
\[P^-_{K,y_1}(x)-\frac{K(M-1)}{2}|x-y|^2+R
\overset{\wt{x}}\eqslantless u(x),~
\forall~x\in B_1.\]
Substituting $x_2\in B_{r/2}(x_0)\subset B_1$ in above,
invoking \eref{eq.x2c} and observing that $y\in \ol{B_{r/8}(x_2)}$,
we deduce that
\[R\leq u(x_2)-P^-_{K,y_1}(x_2)+\frac{K(M-1)}{2}|x_2-y|^2
\leq\left(C+\frac{M-1}{128}\right)Kr^2.\]
On the other hand, we have
\[0\leq u(\wt{x})-P^-_{K,y_1}(\wt{x})
=P^-_{KM,\wt{y}}(\wt{x})-P^-_{K,y_1}(\wt{x})
=-\frac{K(M-1)}{2}|\wt{x}-y|^2+R.\]
Hence
\[|\wt{x}-y|^2
\leq\frac{2}{M-1}\left(C+\frac{M-1}{128}\right)r^2
=\left(\frac{2C}{M-1}+\frac{1}{64}\right)r^2
\leq \frac{1}{16}r^2,\]
provided $M>1$ is sufficiently large.
Thus we obtain $|\wt{x}-y|<r/4$ and
\[|\wt{x}-x_2|
\leq|\wt{x}-y|+|y-x_2|
<r/4+r/8<r/2.\]
Therefore
\begin{equation}\label{eq.tb}
T^-_{KM}(V)\subset B_{r/2}(x_2)\subset B_r(x_0)\subset B_1.
\end{equation}

For each $\wt{y}\in V$,
we see from \eref{eq.ydef} that there exists
$y\in \ol{B_{r/8}(x_2)}$
such that
\[\wt{y}=\frac{M-1}{M}y+\frac{1}{M}y_1.\]
Since $y_1\in \ol{B_1}$
and
$y\in\ol{B_{r/8}(x_2)}\subset B_{r}(x_0)\subset B_1$,
by the convexity of $B_1$, we see that $\wt{y}\in B_1$.
Thus we have $V\subset B_1$ and hence
\begin{equation}\label{eq.tvt}
T^-_{KM}(V)\subset T^-_{KM}(\ol{B_1})=T^-_{KM}.
\end{equation}

Combining \eref{eq.tb} with \eref{eq.tvt} yields
\begin{equation*}
T^-_{KM}(V)\subset B_r(x_0)\cap T^-_{KM},
\end{equation*}
which is exactly the assertion \eref{eq.tvbt}.
\medskip

\emph{Step 3.}
We assert that
\begin{equation}\label{eq.vct}
|V|\leq C\left|T^-_{KM}(V)\right|,
\end{equation}
where $C=C(n,\lambda,\Lambda)>0$.
If we have proved this,
we will conclude from \eref{eq.tvbt} that
\[\left|B_r(x_0)\cap T^-_{KM}\right|
\geq\left|T^-_{KM}(V)\right|
\geq\frac{1}{C}|V|
=\frac{1}{C}\left(\frac{M-1}{8M}\right)^n|B_{r}|
=:\Theta\left|B_r(x_0)\right|,\]
which proves the lemma by taking
$0<\theta=\theta(n,\lambda,\Lambda,\theta_0)
:=\frac{1}{2}\min\{\theta_0,\Theta\}<1$.

Hence we now need only to prove \eref{eq.vct}.
For each $x\in T^-_{KM}(V)$,
there exists a unique $y\in V$ satisfying
\[Du(x)=-KM(x-y)\]
and
\[D^2u(x)\geq-KMI.\]
Consider the mapping $y:x\mapsto y$,
$T^-_{KM}(V)\ra V$, given precisely by
\[y=x+\frac{1}{KM}Du(x).\]
Since
\[|Du(x)|=KM|x-y|\leq2KM,\]
and
\[D_xy=I+\frac{1}{KM}D^2u(x)\geq0,\]
we deduce from \eref{eq.sups} and \eref{eq.tb} that
\begin{eqnarray*}
\lambda\textrm{tr}(D_xy)=\MP_{\lambda,\Lambda}^-(D_xy)
&\leq&\MP_{\lambda,\Lambda}^+(I)
+\frac{1}{KM}\MP_{\lambda,\Lambda}^-(D^2u(x))\\
&\leq& n\Lambda+\frac{1}{KM}\left(|Du(x)|^{\gamma}|f(x)|
+|Du(x)|\right)\\
&\leq& n\Lambda+\frac{2|f(x)|}{(KM)^{1-\gamma}}+2\\
&\leq& C\left(1+\frac{|f(x)|}{(KM)^{1-\gamma}}\right),
\end{eqnarray*}
and
\begin{equation}\label{dxylc}
0\leq\det(D_xy)\leq\left(\frac{\textrm{tr}(D_xy)}{n}\right)^n
\leq C\left(1+\frac{|f(x)|^n}{(KM)^{(1-\gamma)n}}\right),
\end{equation}
where the constants $C$
(and all the $C$'s in the rest of this proof)
depend only on $n,\lambda$ and $\Lambda$.
Hence we can conclude, by the area formula, that
\[|V|\leq\underset{T^-_{KM}(V)}\int\det(D_xy)dx
\leq C\left|T^-_{KM}(V)\right|
+\frac{C}{(KM)^{(1-\gamma)n}}\underset{T^-_{KM}(V)}\int|f(x)|^ndx.\]
In view of the facts that
$T^-_{KM}(V)\subset B_{r}(x_0)\subset B_{2r}(x_1)$
and $\mathcal{M}(|f|^n)(x_1)\leq\epsilon_2K^{(1-\gamma)n}$,
we have
\begin{eqnarray*}
|V|&\leq& C\left|T^-_{KM}(V)\right|
+\frac{C}{(KM)^{(1-\gamma)n}}\underset{B_{2r}(x_1)}\int|f(x)|^ndx\\
&\leq& C\left|T^-_{KM}(V)\right|
+\frac{Cr^n}{K^{(1-\gamma)n}}\cdot\mathcal{M}(|f|^n)(x_1)\\
&\leq& C\left|T^-_{KM}(V)\right|+\epsilon_2Cr^n.
\end{eqnarray*}
Taking $\epsilon_2>0$ to be sufficiently small such that
\[\epsilon_2Cr^n<\frac{1}{2}|V|
=\frac{1}{2}\left(\frac{M-1}{8M}\right)^n|B_1|r^n,\]
we obtain
\[\frac{1}{2}|V|\leq C\left|T^-_{KM}(V)\right|,\]
which implies \eref{eq.vct} and completes the proof of \lref{le.dl}.
\end{proof}

With the \lref{le.dl} in hand,
we can now prove the following measure decay estimate
which concerns the decay of $\left|B_1\setminus T^-_{t}\right|$ in $t$.
\begin{lemma}\label{le.ed}
Let $0\leq\gamma<1$ and $K\geq1$.
Assume that $u\in C^0(\ol{B_1})$ satisfies
\begin{equation*}
|Du|^{-\gamma}\MP_{\lambda,\Lambda}^-(D^2u)-|Du|^{1-\gamma}
\leq f~~\mbox{in}~B_1
\end{equation*}
in the viscosity sense,
where $f\in C^0\cap L^n(B_1)$.
Then there exist constants
$\epsilon_1=\epsilon_1(n,\lambda,\Lambda)>0$,
$\epsilon_2=\epsilon_2(n,\lambda,\Lambda)>0$,
$M=M(n,\lambda,\Lambda)>1$
and $0<\mu_0=\mu_0(n,\lambda,\Lambda)<1$
such that
if $\underset{B_1}{\osc}~u\leq1/8$
and $\norm{f}_{L^n(B_1)}\leq\epsilon_1$,
then
\[\left|B_1\setminus T^-_{KM}\right|
\leq \mu_0\left(\left|B_1\setminus T^-_{K}\right|
+\left|\left\{x\in B_1:\mathcal{M}(|f|^n)(x)>
\epsilon_2 K^{(1-\gamma)n}\right\}\right|\right).\]
\end{lemma}

\begin{proof}
We first prove that
there exists $0<\theta_0=\theta_0(n,\lambda,\Lambda)<1$
such that
\begin{equation}\label{eq.tt0}
\left|B_1\cap T^-_K\cap\left\{x\in B_1:\mathcal{M}(|f|^n)(x)
\leq\epsilon_2K^{(1-\gamma)n}\right\}\right|\geq\theta_0|B_1|.
\end{equation}

For each $\wt{x}\in T^-_K\left(\ol{B_{1/2}}\right)$,
there exists $\wt{y}\in \ol{B_{1/2}}$
such that
\[u(\wt{x})+\frac{K}{2}|\wt{x}-\wt{y}|^2
=\underset{x\in \ol{B_1}}\min
\left(u(x)+\frac{K}{2}|x-\wt{y}|^2\right)=:m(\wt{y}).\]
Hence we conclude that
$-\frac{K}{2}|x-\wt{y}|^2+m(\wt{y})
\overset{\wt{x}}\eqslantless u$ in $B_1$.
In particular, we have
$m(\wt{y})\leq u(\wt{y})$
and
$-\frac{K}{2}|\wt{x}-\wt{y}|^2+m(\wt{y})=u(\wt{x})$.
Subtracting one from the other, we deduce that
\[\frac{K}{2}|\wt{x}-\wt{y}|^2
\leq u(\wt{y})-u(\wt{x})
\leq\underset{B_1}{\osc}~u\leq1/8,\]
which implies $|\wt{x}-\wt{y}|\leq1/2$.
Thus $\wt{x}\in B_1$ and hence
\begin{equation}\label{eq.tkb}
T^-_K\left(\ol{B_{1/2}}\right)\subset B_1.
\end{equation}

Consider, as in the proof of \lref{le.dl},
the mapping $\wt{y}:\wt{x}\mapsto\wt{y}$,
$T^-_K\left(\ol{B_{1/2}}\right)\ra \ol{B_{1/2}}$,
given by
\[\wt{y}=\wt{x}+\frac{1}{K}Du(\wt{x}).\]
Since $|Du(\wt{x})|=K|\wt{x}-\wt{y}|\leq2K$,
we conclude, as in \eref{dxylc}, that
\[0\leq\det(D_{\wt{x}}\wt{y})
\leq C\left(1+\frac{|f(\wt{x})|^n}{K^{(1-\gamma)n}}\right)
\leq C(1+|f(\wt{x})|^n),\]
where $C=C(n,\lambda,\Lambda)>0$.
Thus it follows from the area formula and \eref{eq.tkb} that
\[2^{-n}|B_1|=\left|\ol{B_{1/2}}\right|
\leq \underset{T^-_K\left(\ol{B_{1/2}}\right)}\int
\det(D_{\wt{x}}\wt{y})d\wt{x}
\leq C\left|T^-_K\left(\ol{B_{1/2}}\right)\right|+C\underset{B_1}\int|f|^n.\]
Let $\norm{f}^n_{L^n(B_1)}\leq 2^{-(n+1)}|B_1|\cdot C^{-1}$.
We obtain
\[2^{-(n+1)}|B_1|\leq C\left|T^-_K\left(\ol{B_{1/2}}\right)\right|
=C|B_1\cap T^-_K\left(\ol{B_{1/2}}\right)|\leq C|B_1\cap T^-_K|,\]
where we have used \eref{eq.tkb} again. Thus
\begin{equation}\label{eq.tt1}
|B_1\cap T^-_K|\geq2^{-(n+1)}C^{-1}|B_1|=:\theta_1|B_1|,
\end{equation}
where $0<\theta_1=\theta_1(n,\lambda,\Lambda)<1$.

On the other hand, by the weak type (1,1) property, we have
\begin{eqnarray*}
&~&\left|\left\{x\in B_1:\mathcal{M}(|f|^n)(x)>
\epsilon_2K^{(1-\gamma)n}\right\}\right|\\
&\leq& C(n)\left(\epsilon_2K^{(1-\gamma)n}\right)^{-1}\norm{|f|^n}_{L^1(B_1)}
\leq C(n){\epsilon_2}^{-1}\norm{f}^n_{L^n(B_1)}.
\end{eqnarray*}
Let
\[\norm{f}^n_{L^n(B_1)}\leq\epsilon_1^n
:=\min\left\{2^{-(n+1)}|B_1|\cdot C^{-1},~
\frac{\theta_1}{2}|B_1|\cdot C(n)^{-1}\epsilon_2\right\}.\]
We obtain
\[\left|\left\{x\in B_1:\mathcal{M}(|f|^n)(x)>
\epsilon_2K^{(1-\gamma)n}\right\}\right|
\leq \frac{\theta_1}{2}|B_1|,\]
and hence
\begin{equation}\label{eq.mftt1}
\left|\left\{x\in B_1:\mathcal{M}(|f|^n)(x)\leq
\epsilon_2K^{(1-\gamma)n}\right\}\right|
\geq \left(1-\frac{\theta_1}{2}\right)|B_1|.
\end{equation}
Combining \eref{eq.tt1} with \eref{eq.mftt1}, we conclude that
\[\left|B_1\cap T^-_K\cap
\left\{x\in B_1:\mathcal{M}(|f|^n)(x)\leq\epsilon_2K^{(1-\gamma)n}\right\}\right|
\geq\frac{\theta_1}{2}|B_1|=:\theta_0|B_1|,\]
which is \eref{eq.tt0}.

\medskip

Now according to \eref{eq.tt0} and the conclusion of \lref{le.dl},
applying the covering lemma, \lref{le.cl}, to
\[E:=T^-_K\cap\left\{x\in B_1:\mathcal{M}(|f|^n)(x)\leq
\epsilon_2K^{(1-\gamma)n}\right\}\]
and $F:=T^-_{KM}$, we thus conclude that
\begin{eqnarray*}
\left|B_1\setminus T^-_{KM}\right|
&=&|B_1\setminus F|
\leq\left(1-\frac{\Theta-\theta}{5^n}\right)|B_1\setminus E|
=:\mu_0|B_1\setminus E|\\
&\leq&\mu_0\left(|B_1\setminus T^-_K|
+\left|\left\{x\in B_1:\mathcal{M}(|f|^n)(x)>
\epsilon_2K^{(1-\gamma)n}\right\}\right|\right).
\end{eqnarray*}
This finishes the proof of \lref{le.ed}.
\end{proof}

\begin{corollary}\label{co.ed}
Under the assumptions of \lref{le.ed}, we have
\[\left|B_1\setminus T^-_{M^k}\right|
\leq C\mu^k \quad (k=0,1,2,...),\]
where $C=C(n,\lambda,\Lambda,\gamma)>0$
and $0<\mu=\mu(n,\lambda,\Lambda,\gamma)<1$.
\end{corollary}

\begin{proof}
Let
\[\alpha_k:=\left|B_1\setminus T^-_{M^k}\right|\]
and
\[\beta_k:=\left|\left\{x\in B_1:\mathcal{M}(|f|^n)(x)
>\epsilon_2M^{(1-\gamma)nk}\right\}\right|.\]
Applying \lref{le.ed} to $K=M^k$, $\forall k=0,1,2,...$, we obtain
\[\alpha_{k+1}\leq\mu_0(\alpha_k+\beta_k)
\quad (\forall k=0,1,2,...).\]

Using the weak type (1,1) property,
recalling $\norm{f}_{L^n(B_1)}\leq\epsilon_1$,
and remembering that $\epsilon_1$ and $\epsilon_2$
depend only on $n,\lambda$ and $\Lambda$, we conclude that
\[
\beta_k\leq C(n)\left(\epsilon_2M^{(1-\gamma)nk}\right)^{-1}
\norm{f}^n_{L^n(B_1)}
\leq C_0\left(M^{-(1-\gamma)n}\right)^{k}|B_1|,
\]
where $C_0=C_0(n,\lambda,\Lambda)>0$.
Thus we have
\begin{eqnarray*}
\sum_{i=0}^{k-1}\mu_0^{k-i}\beta_i
&\leq& C_0|B_1|\sum_{i=0}^{k-1}\mu_0^{k-i}\left(M^{-(1-\gamma)n}\right)^{i}\\
&\leq& C_0|B_1|\sum_{i=0}^{k-1}\mu_1^k
=C_0|B_1|k\mu_1^k \quad (\forall k=1,2,3,...),
\end{eqnarray*}
and hence
\[\alpha_k
\leq\mu_0^k|B_1|+\sum_{i=0}^{k-1}\mu_0^{k-i}\beta_i
\leq(1+C_0k)|B_1|\mu_1^k
\leq C\mu^k  \quad (\forall k=1,2,3,...),\]
where
$\mu_1:=\max\left\{\mu_0,M^{-(1-\gamma)n}\right\}\in(0,1)$,
$\mu:=(1+\mu_1)/2\in(\mu_1,1)$ and
$C:=\max_{t\in \R}\left\{\mu^{-t}\mu_1^{t}(1+C_0t)|B_1|\right\}
\in(0,+\infty)$
are all constants depending only on $n,\lambda,\Lambda$ and $\gamma$.
This finishes the proof of \textsl{Corollary \ref{co.ed}}.
\end{proof}

\subsection{Proof of \lref{le.w2dc}}\label{sse.pf-le}
\quad

From the above \cref{co.ed}, \lref{le.w2dc} follows easily.
\begin{proof}[Proof of \lref{le.w2dc}]
Since $\MP^+(D^2u)=-\MP^-(D^2(-u))$,
by the second inequality of \eref{eq.ineqs},
it is clear that $-u$ satisfies
all the assumptions of \lref{le.ed}
and hence of \cref{co.ed}.
Since, by definition,
$T^+_\kappa\left(u,\ol{B_1}\right)
=T^-_\kappa\left(-u,\ol{B_1}\right)$,
applying \cref{co.ed} to $-u$, we get
$\left|B_1\setminus T^+_{M^k}\right|\leq C\mu^k$,
$\forall k\in\Z^+$.
Thus
\begin{eqnarray*}
\left|B_1\setminus T_{M^k}\right|
=\left|B_1\setminus \left(T^-_{M^k}\cap T^+_{M^k}\right)\right|
\leq \left|B_1\setminus T^-_{M^k}\right|
+\left|B_1\setminus T^+_{M^k}\right|
\leq C\mu^k,~\forall k\in\Z^+;
\end{eqnarray*}
and hence
\begin{equation}\label{eq.tt}
\left|B_1\setminus T_{t}\right|\leq Ct^{-\sigma},
~\forall t>0,
\end{equation}
where $\sigma:=-\log_M\mu$.
Invoking \lref{le.lp}, we deduce that
$\norm{D^2u}_{L^\delta(B_1)}\leq C$
(see also \cite[proposition 1.1]{CC}).
By the interpolation theorem (see \cite[Theorem 7.28]{GT}),
we thus obtain $\norm{u}_{W^{2,\delta}(B_1)}\leq C$.
This completes the proof of \lref{le.w2dc}.
\end{proof}

\begin{remark}
For heuristic purpose, we give for $u\in C^2(B_1)$
the simple and full details
of deducing $\norm{D^2u}_{L^\delta(B_1)}\leq C$ from \eref{eq.tt}.
Indeed, since
\[B_1\cap T_t\subset\left\{x\in B_1:-tI\leq D^2u(x)\leq tI\right\}
\subset\left\{x\in B_1:\left|D^2u(x)\right|\leq \sqrt{n}t\right\},\]
we have
\[\left\{x\in B_1:\left|D^2u(x)\right|>\sqrt{n}t\right\}
\subset B_1\setminus T_t.\]
Hence
\[\left|\left\{
x\in B_1:|D^2u(x)|>\sqrt{n}t\right\}\right|
\leq|B_1\setminus T_t|\leq Ct^{-\sigma}.\]
Using \lref{le.lp}, we thus obtain
\begin{eqnarray*}
\norm{D^2u}^\delta_{L^\delta(B_1)}
&\leq& C(n,M,\delta)\left(|B_1|+\underset{k}\sum M^{\delta k}
\left|\left\{x\in B_1:\left|D^2u(x)\right|>\sqrt{n}M^k\right\}\right|\right)\\
&\leq& C(n,M,\delta)\left(|B_1|+C\underset{k}\sum M^{(\delta-\sigma)k}\right)
\leq C(n,\lambda,\Lambda,\delta).
\end{eqnarray*}
\end{remark}


\end{document}